\documentclass{amsart}
\usepackage{amssymb,mathtools}
\usepackage[british]{babel}
\usepackage{enumitem}
\usepackage[latin1]{inputenc}
\usepackage{url}

\newtheorem{theorem}{Theorem}
\numberwithin{theorem}{section}
\newtheorem{corollary}[theorem]{Corollary}
\newtheorem{lemma}[theorem]{Lemma}
\newtheorem{proposition}[theorem]{Proposition}
\theoremstyle{definition}
\newtheorem{definition}[theorem]{Definition}
\newtheorem{remark}[theorem]{Remark}

\newcommand{\supp}{\operatorname{supp}}
\newcommand{\rng}{\operatorname{rng}}
\newcommand{\leqf}{\leq^{\operatorname{fin}}}

\newcommand{\T}{\mathcal T}
\newcommand{\po}{\operatorname{PO}}

\newcommand{\nf}{\mathrel{=_{\operatorname{NF}}}}

\newcommand{\ot}{\mathbb{T}}
\newcommand{\rt}{\operatorname{root}}

\title[From Kruskal's theorem to Friedman's gap condition]{From Kruskal's theorem\\ to Friedman's gap condition}
\author{Anton Freund}

\begin{document}

\begin{abstract}
Harvey Friedman's gap condition on embeddings of finite labelled trees plays an important role in combinatorics (proof of the graph minor theorem) and mathematical logic (strong independence results). In the present paper we show that the gap condition can be reconstructed from a small number of well-motivated building blocks: it arises via iterated applications of a uniform Kruskal theorem.
\end{abstract}

\maketitle

\section{Introduction}

In this paper, a tree is a finite partial order $T=(T,\leq_T)$ such that
\begin{itemize}
 \item the order $T$ has a unique minimal element $\langle\rangle$, called the root of~$T$, and
 \item for each $t\in T$, the set $\{s\in T\,|\,s\leq_T t\}$ is linearly ordered by $\leq_T$.
\end{itemize}
For each pair of elements $s,t\in T$ there is a $\leq_T$-maximal element $s\wedge t\in T$ with $s\wedge t\leq_T s$ and $s\wedge t\leq_T t$. An embedding of trees is given by a function $f:S\to T$ that satisfies
\begin{equation*}
 f(s\wedge t)=f(s)\wedge f(t)
\end{equation*}
for all~$s,t\in S$. Since $s\leq_S t$ is equivalent to $s\wedge t=s$, this entails that $f$ is an embedding of partial orders and in particular injective. Kruskal's theorem~\cite{kruskal60} asserts the following: For any infinite sequence $T_0,T_1,\dots$ of finite trees, there are indices $i<j$ such that $T_i$ can be embedded into~$T_j$.

Let us point out that Kruskal's theorem has important implications for theoretical computer science (cf.~the work of N.~Dershowitz~\cite{dershowitz87}) and mathematical logic. Concerning the latter, a classical result of D.~Schmidt~\cite{schmidt-habil} and H.~Friedman~\cite{simpson85} shows that Kruskal's theorem cannot be proved in $\mathbf{ATR_0}$, a relatively strong axiom system that is associated with the predicative foundations of mathematics (see~\cite{gallier-kruskal,simpson09} for detailed explanations). The precise logical strength of Kruskal's theorem has been determined by M.~Rathjen and A.~Weiermann~\cite{rathjen-weiermann-kruskal}.

By an $n$-tree we mean a tree~$T$ together with a function $l:T\to\{0,\dots,n-1\}$. An embedding between $n$-trees $(S,l)$ and $(T,l')$ is given by an embedding $f:S\to T$ of trees that satisfies the following conditions:
\begin{enumerate}[label=(\roman*)]
\item We have $l'(f(s))=l(s)$ for any $s\in S$.
\item If $t$ is an immediate successor of~$r\in S$ (i.\,e.~if $t$ is $\leq_S$-minimal with $r<_S t$) and we have $f(r)<_T s<_T f(t)$, then we have $l'(s)\geq l'(f(t))=l(t)$.
\item If we have $s<_T f(\langle\rangle)$, then we have $l'(s)\geq l'(f(\langle\rangle))=l(\langle\rangle)$.
\end{enumerate}
Part~(ii) and~(iii) constitute the famous gap condition due to H.~Friedman~\cite{simpson85}. More precisely, part~(ii) on its own is known as the weak gap condition. In the present paper we are only concerned with the strong gap condition, which is the conjunction of~(ii) and~(iii). The following result is known as Friedman's theorem: For each number $n$ and any infinite sequence $T_0,T_1,\dots$ of finite $n$-trees, there is an embedding $T_i\to T_j$ of $n$-trees for some indices~$i<j$.

Friedman's theorem plays a role in N.~Robertson and P.~Seymour's proof of their famous graph minor theorem. In fact, Friedman, Robertson and Seymour~\cite{friedman-robertson-seymour} have shown that Friedman's theorem is equivalent to the graph minor theorem for graphs of bounded tree-width, over the weak base theory~$\mathbf{RCA_0}$. From the viewpoint of mathematical logic it is very significant that Friedman's theorem is unprovable in~$\mathbf{\Pi^1_1}\textbf{-}\mathbf{CA_0}$, which is even stronger than the axiom system $\mathbf{ATR_0}$ mentioned above.

The present paper shows that Friedman's gap condition results from iterated applications of a uniform Kruskal theorem. This provides a systematic and transparent reconstruction of the gap condition, which may otherwise feel ad hoc. Furthermore, our reconstruction prepares the computation of maximal order types, as begun by J.~van der Meeren, M.~Rathjen and A.~Weiermann~\cite{meeren-thesis,MRW-Veblen,MRW-Bachmann,MRW-linear}.

Let us explain the uniform Kruskal theorem that was mentioned in the previous paragraph. Given a partial order $X$, we write $W(X)$ for the set of finite multisets with elements from~$X$. Such a multiset can be written as $[x_0,\dots,x_{n-1}]$, where the multiplicity of the entries is relevant but the order is not. To define a partial order on $W(X)$, we declare that $[x_0,\dots,x_{m-1}]\leq_{W(X)}[y_0,\dots,y_{n-1}]$ holds if, and only if, there is an injection $h:\{0,\dots,m-1\}\to\{0,\dots,n-1\}$ such that we have $x_i\leq_X y_{h(i)}$ for all~$i<m$. Write $\T W$ for the set of trees, where isomorphic trees are identified. We get a bijection
\begin{equation*}
\kappa:W(\T W)\to\T W
\end{equation*}
if we define $\kappa([T_0,\dots,T_{n-1}])$ as the tree in which the root has immediate subtrees~$T_0,\dots,T_{n-1}$. Indeed, the set $\T W$ can be charaterized as the initial fixed point of the transformation~$W$. For $S,T\in\T W$ we write $S\leq_{\T W}T$ if there is an embedding $S\to T$ of trees. This relation can also be reconstructed in terms of the order on multisets: Writing $[X]^{<\omega}$ for the set of finite subsets of~$X$, we define a family of functions $\supp^W_X:W(X)\to[X]^{<\omega}$ by setting
\begin{equation*}
\supp^W_X([x_0,\dots,x_{n-1}])=\{x_0,\dots,x_{n-1}\}.
\end{equation*}
For multisets $\sigma$ and $\tau$ in $W(\T W)$ one can verify
\begin{equation}\tag{$\star$}\label{eq:order-TW}
\kappa(\sigma)\leq_{\T W}\kappa(\tau)\,\Leftrightarrow\,(\sigma\leq_{W(\T W)}\tau\text{ or }\kappa(\sigma)\leq_{\T W}T\text{ for some $T\in\supp^W_{\T W}(\tau)$}).
\end{equation}
Indeed, the first disjunct on the right corresponds to an embedding $\kappa(\sigma)\to\kappa(\tau)$ that maps the root to the root, and immediate subtrees to immediate subtrees. The second disjunct corresponds to an embedding that maps all of $\kappa(\sigma)$ into one immediate subtree of~$\kappa(\tau)$.

A PO-dilator is a particularly uniform transformation~$W$ of partial orders that comes with a family of functions $\supp^W_X:W(X)\to[X]^{<\omega}$. In Section~\ref{sect:relativized-kruskal} we will recall the precise definition, as well as a normality condition for PO-dilators. For any normal PO-dilator~$W$ one can construct a ``Kruskal fixed point" $\T W$ that is partially ordered according to~(\ref{eq:order-TW}). Recall that a partial order~$X$ is a well partial order if any infinite sequences $x_0,x_1,\dots$ in $X$ admits indices $i<j$ with $x_i\leq_X x_j$. A PO-dilator~$W$ is called a WPO-dilator if $W(X)$ is a well partial order whenever the same holds for~$X$. The uniform Kruskal theorem asserts that $\T W$ is a well partial order for any normal WPO-dilator~$W$. In the previous paragraph we have see that the usual Kruskal theorem arises as a special case. It is instructive to check that Higman's lemma is another special case (take $W(X)=1+Z\times X$ to generate finite lists with entries in~$Z$). As shown by A.~Freund, M.~Rathjen and A.~Weiermann~\cite{frw-kruskal}, the uniform Kruskal theorem is equivalent to $\Pi^1_1$-comprehension (the main axiom of~$\mathbf{\Pi^1_1}\textbf{-}\mathbf{CA_0}$), over $\mathbf{RCA_0}$ together with the chain antichain principle. This result builds on a corresponding equivalence in the context of linear orders, which is due to the present author~\cite{freund-equivalence,freund-categorical,freund-computable}.

In this paper we show how the construction of $\T W$ can be relativized to a given partial order~$X$. The result is a partial order $\T W(X)$ with a bijection
\begin{equation*}
X\sqcup W(\T W(X))\to\T W(X).
\end{equation*}
The point of the relativization is that $\T W$ becomes a transformation of partial orders. We will show that $\T W$ can itself be equipped with the structure of a normal PO-dilator, which we call the Kruskal derivative of~$W$. The axiom of~$\Pi^1_1$-comprehension is still equivalent to the principle that $\T W$ is a normal WPO-dilator whenever the same holds for~$W$. This principle will also be referred to as the uniform Kruskal theorem. 

Our main aim is to reconstruct Friedman's gap condition by taking iterated Kruskal derivatives. In the following we write $M(X)$ for the set of multisets with elements from~$X$. One can equip $M$ with the structure of a normal WPO-dilator. Our reconstruction of Friedman's gap condition proceeds via the following steps:
\begin{enumerate}
\item Start with the normal WPO-dilator~$\ot_0$ given by $\ot_0(X)=X$.
\item Assuming that the normal WPO-dilator $\ot_n$ is already constructed, define the normal WPO-dilator $\ot_{n+1}^-$ as the Kruskal derivative of~$M\circ \ot_n$.
\item Define the normal WPO-dilator $\ot_{n+1}$ as the composition $\ot_n\circ \ot_{n+1}^-$.
\item Verify that $\ot_n(\emptyset)$ is isomorphic to the set of $n$-trees, ordered according to Friedman's strong gap condition.
\end{enumerate}
In particular, the statement that $\ot_n$ is a WPO-dilator follows from $n$ applications of the uniform Kruskal theorem. If $\Pi^1_2$-induction is available, then one can conclude that the statement holds for all $n\in\mathbb N$. This helps to explain why~$\mathbf{\Pi^1_1}\textbf{-}\mathbf{CA_0}$ does not prove that~$\ot_n(\emptyset)$ is a well partial order for every~$n\in\mathbb N$, even though it proves the statement for each fixed number.

To conclude this introduction, we discuss related results from the literature. The original proof of Friedman's theorem~\cite{simpson85} involves iterated applications of the minimal bad sequence argument, which broadly resemble steps~(2) and~(3) above. It does not, however, translate these iterations into a recursive definition of the gap condition. Instead, it seems that the latter was originally motivated by certain ordinal notation systems. Our transformation of~$W$ into $\T W$ is very similar to a construction by~R.~Hasegawa~\cite{hasegawa97}. Without giving a detailed proof, Hasegawa even states that iterations of the construction lead to a variant of Friedman's gap condition for trees with edge labels. Van der Meeren, Rathjen and Weiermann have reconstructed suborders of the trees with gap condition, with certain restrictions on the distribution of labels (see e.\,g.~\cite[Definition~16]{MRW-Bachmann} and~\cite[Definition~12]{MRW-linear}). As far as we know, the present paper is the first to give a detailed reconstruction of the gap condition in its original form.

\textbf{Acknowledgements.} I am very grateful to Jeroen van der Meeren, Michael Rathjen and Andreas Weiermann. I owe them many of the ideas that were fundamental for the present paper.

\section{Relativized Kruskal fixed points}\label{sect:relativized-kruskal}

In this section we recall the definition of normal PO-dilator. We then construct the relativized Kruskal fixed points $\T W(X)$ that were mentioned in the introduction. We will introduce these fixed points in terms of notation systems. A more semantic characterization will follow in the next section.

Jean-Yves Girard~\cite{girard-pi2} has introduced dilators as particularly uniform transformations of linear orders. A corresponding definition for partial orders has been given by Freund, Rathjen and Weiermann~\cite{frw-kruskal}. In order to recall the precise definition, we need some terminology: A function $f:X\to Y$ between partial orders is called a quasi embedding if $f(x)\leq_Y f(y)$ implies $x\leq_X y$. If the converse implication holds as well, then we have an embedding. The category~$\po$ consists of the partial orders as objects and the quasi embeddings as morphisms. We say that a functor $W:\po\to\po$ preserves embeddings if $W(f):W(X)\to W(Y)$ is an embedding whenever the same holds for~$f:X\to Y$. As in the introduction, we write $[X]^{<\omega}$ for the set of finite subsets of a given set~$X$. To turn $[\cdot]^{<\omega}$ into a functor, we define
\begin{equation*}
 [f]^{<\omega}(a)=\{f(x)\,|\,x\in a\}\in[Y]^{<\omega}\qquad\text{for $f:X\to Y$ and $a\in[X]^{<\omega}$}.
\end{equation*}
We also apply $[\cdot]^{<\omega}$ to partial orders, omitting the forgetful functor to the underlying set. Conversely, subsets of partial orders are often considered as suborders.

\begin{definition}\label{def:po-dilator}
 A PO-dilator consists of
 \begin{enumerate}[label=(\roman*)]
  \item a functor $W:\po\to\po$ that preserves embeddings and
  \item a natural transformation $\supp^W:W\Rightarrow[\cdot]^{<\omega}$ that satisfies the following support condition: Given any embedding $f:X\to Y$ of partial orders, the embedding $W(f):W(X)\to W(Y)$ has range
  \begin{equation*}
   \rng(W(f))=\{\sigma\in W(Y)\,|\,\supp^W_Y(\sigma)\subseteq\rng(f)\}.
  \end{equation*}
 \end{enumerate}
 If $W(X)$ is a well partial order (wpo) for any wpo~$X$, then~$W$ is a WPO-dilator.
\end{definition}

The reader may have observed that the previous definition focuses on embeddings rather than quasi embeddings. The latter are important for applications to the theory of well partial orders (see e.\,g.~\cite{frw-kruskal}). Also, the inclusion $\subseteq$ in part~(ii) of the definition is automatic, since the naturality of supports yields
\begin{equation*}
 \supp^W_Y(W(f)(\sigma_0))=[f]^{<\omega}(\supp^W_X(\sigma_0))\subseteq\rng(f).
\end{equation*}
When the partial order~$X$ is clear from the context, then $\iota_a:a\hookrightarrow X$ denotes the inclusion of a suborder~$a\subseteq X$. For $\sigma\in W(X)$ we write
\begin{equation*}
 \sigma\nf W(\iota_a)(\sigma_0)\qquad\text{with $a\in[X]^{<\omega}$ and $\sigma_0\in W(a)$}
\end{equation*}
if the equality holds and we have $\supp^W_a(\sigma_0)=a$. The latter is a uniqueness condition, which is required for the following result:

\begin{lemma}\label{lem:normal-form}
 Consider a PO-dilator~$W$ and a partial order~$X$. Any $\sigma\in W(X)$ has a unique normal form $\sigma\nf W(\iota_a)(\sigma_0)$. For the latter we have $a=\supp^W_X(\sigma)$.
\end{lemma}
\begin{proof}
 Let us first show that the representation is unique. Since $\supp^W$ is natural, we can observe that $\sigma\nf W(\iota_a)(\sigma_0)$ entails $\supp^W_a(\sigma_0)=a$ and hence
 \begin{equation*}
  \supp^W_X(\sigma)=\supp^W_X(W(\iota_a)(\sigma_0))=[\iota_a]^{<\omega}(\supp^W_a(\sigma_0))=[\iota_a]^{<\omega}(a)=a.
 \end{equation*}
 This means that $a$ is determined by~$\sigma$. Just as any embedding, the function $W(\iota_a)$ is injective. Hence $\sigma_0$ is uniquely determined as well. In order to prove existence, we put $a=\supp^W_X(\sigma)$. Then we have $\supp^W_X(\sigma)\subseteq a=\rng(\iota_a)$, so that the support condition yields $\sigma=W(\iota_a)(\sigma_0)$ for some $\sigma_0\in W(a)$. We also have
 \begin{equation*}
  a=\supp^W_X(\sigma)=[\iota_a]^{<\omega}(\supp^W_a(\sigma_0)).
 \end{equation*}
 This implies $\supp^W_a(\sigma_0)=a$ and hence $\sigma\nf W(\iota_a)(\sigma_0)$.
\end{proof}

The normal forms from the previous lemma can be used to represent PO-dilators in second order arithmetic, as worked out in~\cite{frw-kruskal}. In the present paper we do not work within a particular meta theory. Given a partial order~$X$, we define a quasi order $\leqf_X$ on the set $[X]^{<\omega}$ by stipulating
\begin{equation*}
 a\leqf_X b\quad\Leftrightarrow\quad\text{for any $x\in a$ there is a $y\in b$ with $x\leq_X y$}.
\end{equation*}
We will write $a\leqf_X y$ rather than $a\leqf_X\{y\}$ in the case of a singleton. The following normality condition turns out to be crucial:

\begin{definition}\label{def:normal}
 A PO-dilator $W$ is called normal if we have
 \begin{equation*}
  \sigma\leq_{W(X)}\tau\quad\Rightarrow\quad\supp^W_X(\sigma)\leqf_X\supp^W_X(\tau),
 \end{equation*}
 for any partial order~$X$ and arbitrary elements~$\sigma,\tau\in W(X)$.
\end{definition}

In many applications, the elements $\sigma,\tau\in W(X)$ are finite structures with labels in~$X$. Then the inequality $\supp^W_X(\sigma)\leqf_X\supp^W_X(\tau)$ corresponds to the condition that each label is mapped to a bigger one. In~\cite{frw-kruskal}, the Kruskal fixed point $\T W$ of a normal PO-dilator has been generated by the following inductive clause:
\begin{itemize}
 \item Assuming that we have already generated a finite suborder $a\subseteq\T W$, we add a term $\circ(a,\sigma)\in\T W$ for each element $\sigma\in W(a)$ with $\supp^W_a(\sigma)=a$.
\end{itemize}
The point is that one can now define a bijection~$\kappa:W(\T W)\to\T W$ by stipulating $\kappa(\sigma)=\circ(a,\sigma_0)$ for $\sigma\nf W(\iota_a)(\sigma_0)$. We will relativize the construction by including constant symbols $\overline x\in\T W(X)$ for elements $x\in X$ of a given partial order. At various places in the following definition, we require that $\leq_{\T W(X)}$ is a partial order on certain subsets of~$\T W(X)$. We will later show that all of $\T W(X)$ is partially ordered by $\leq_{\T W(X)}$, so that these requirements become redundant. A~more detailed justification of the following recursion can be found below.

\begin{definition}\label{def:TW(X)}
 Consider a normal PO-dilator~$W$. For each partial order~$X$ we define a set $\T W(X)$ of terms and a binary relation $\leq_{\T W(X)}$ on this set by simultaneous recursion. The set $\T W(X)$ is generated by the following clauses:
 \begin{enumerate}[label=(\roman*)]
  \item For each element $x\in X$ we have a term $\overline x\in\T W(X)$.
  \item Given a finite set $a\subseteq\T W(X)$ that is partially ordered by $\leq_{\T W(X)}$, we add a term $\circ(a,\sigma)\in\T W(X)$ for each $\sigma\in W(a)$ with $\supp^W_a(\sigma)=a$.
 \end{enumerate}
 For $s,t\in\T W(X)$ we stipulate that $s\leq_{\T W(X)}t$ holds if, and only if, one of the following clauses applies:
 \begin{enumerate}[label=(\roman*')]
  \item We have $s=\overline x$ and $t=\overline y$ with $x\leq_X y$.
  \item We have $t=\circ(b,\tau)$ and $s\leq_{\T W(X)}t'$ for some $t'\in b$ (where $s$ can be of the form $\overline x$ or $\circ(a,\sigma)$).
  \item We have $s=\circ(a,\sigma)$ and $t=\circ(b,\tau)$, the restriction of $\leq_{\T W(X)}$ to $a\cup b$ is a partial order, and we have
  \begin{equation*}
   W(\iota_a)(\sigma)\leq_{W(a\cup b)} W(\iota_b)(\tau),
  \end{equation*}
  where $\iota_a:a\hookrightarrow a\cup b$ and $\iota_b:b\hookrightarrow a\cup b$ are the inclusions.
 \end{enumerate}
\end{definition}

To justify the recursion in detail, one can argue as follows: First generate a set $\T_0W(X)\supseteq\T W(X)$ by including all terms $\circ(a,\sigma)$ for finite $a\subseteq\T_0W(X)$, where $a$ is not assumed to be ordered and $\sigma\in W(a)$ holds with respect to some partial order on~$a$. Then define a length function $l_X:\T_0W(X)\to\mathbb N$ by the recursive clauses
\begin{equation*}
 l_X(\overline x)=0,\qquad l_X(\circ(a,\sigma))=1+\textstyle\sum_{r\in a}2\cdot l_X(r).
\end{equation*}
One can now decide $r\in\T W(X)$ and $s\leq_{\T W(X)}t$ by simultaneous recursion on~$l_X(r)$ and $l_X(s)+l_X(t)$. As an example, we consider the case of $r=\circ(a,\sigma)$. For $s,t\in a$ we have $l_X(s)+l_X(t)<l_X(r)$, even when $s$ and $t$ are the same term (due to the factor~$2$ above). Recursively, we can thus determine the restriction of $\leq_{\T W(X)}$ to~$a$. If the latter is a partial order, we check whether $\sigma\in W(a)$ and $\supp^W_a(\sigma)=a$ hold with respect to this order. When this is the case, we have $r\in\T W(X)$. In addition to the length functions, we need the height functions $h_X:\T W(X)\to\mathbb N$ given by
\begin{equation*}
 h_X(\overline x)=0,\qquad h_X(\circ(a,\sigma))=\max(\{0\}\cup\{h_X(r)+1\,|\,r\in a\}).
\end{equation*}
When there us no danger of confusion, we sometimes omit the index~$X$. The following important observation relies on the assumption that $W$ is normal. It confirms the intuition that~$\T W(X)$ can be seen as a tree-like structure.

\begin{lemma}\label{lem:ineq-heights}
 Consider a normal PO-dilator~$W$ and a partial order~$X$. For any elements $s,t\in\T W(X)$, the inequality $s\leq_{\T W(X)}t$ implies $h_X(s)\leq h_X(t)$.
\end{lemma}
\begin{proof}
 One argues by induction on~$l(s)+l(t)$. The case of~$s=\overline x$ and~$t=\overline y$ is immediate. The remaining cases are similar to the proof of~\cite[Lemma~3.5]{frw-kruskal}. First assume that
 \begin{equation*}
  s\leq_{\T W(X)}\circ(b,\tau)=t
 \end{equation*}
 holds because we have $s\leq_{\T W(X)}t'$ for some $t'\in b$. In view of $l(t')<l(t)$ the induction hypothesis yields $h(s)\leq h(t')<h(t)$. Now assume that
 \begin{equation*}
  s=\circ(a,\sigma)\leq_{\T W(X)}\circ(b,\tau)=t
 \end{equation*}
 holds because of~$W(\iota_a)(\sigma)\leq_{W(a\cup b)} W(\iota_b)(\tau)$. Since $W$ is normal, we get
 \begin{equation*}
  a=[\iota_a]^{<\omega}(\supp^W_a(\sigma))=\supp^W_{a\cup b}(W(\iota_a)(\sigma))\leqf_{\T W(X)}\supp^W_{a\cup b}(W(\iota_b)(\tau))=b.
 \end{equation*}
 Given any $s'\in a$ we thus have $s'\leq_{\T W(X)}t'$ for some $t'\in b$. By induction hypothesis we obtain $h(s')\leq h(t')<h(t)$. As $s'\in a$ was arbitrary, this yields $h(s)\leq h(t)$.
\end{proof}

The proof of the following result is similar to the one of~\cite[Proposition~3.6]{frw-kruskal}. Since the present notation is somewhat different, we reproduce the proof for the reader's convenience.

\begin{proposition}\label{prop:fixed-points-partial-order}
 The relation $\leq_{\T W(X)}$ is a partial order on $\T W(X)$, for any normal PO-dilator~$W$ and any partial order~$X$.
\end{proposition}
\begin{proof}
 One uses simultaneous induction on~$n$ to establish
 \begin{align*}
  r&\leq_{\T W(X)}r && \text{for $l(r)\leq n$},\\
  (s\leq_{\T W(X)}t\,\land\, t&\leq_{\T W(X)}s)\,\Rightarrow\,s=t && \text{for $l(s)+l(t)\leq n$},\\
  (r\leq_{\T W(X)}s\,\land\, s&\leq_{\T W(X)} t)\,\Rightarrow\,r\leq_{\T W(X)}t && \text{for $l(r)+l(s)+l(t)\leq n$}.
 \end{align*}
 Reflexivity is readily verified. Concerning antisymmetry, we consider the case where $s\leq_{\T W(X)}\circ(b,\tau)=t$ holds because we have $s\leq_{\T W(X)}t'$ for some $t'\in b$. By the previous lemma we get $h(s)\leq h(t')<h(t)$, which makes $t\leq_{\T W(X)}s$ impossible. Still for antisymmetry, we also consider the case where $s\leq_{\T W(X)}t\leq_{\T W(X)}s$ with $s=\circ(a,\sigma)$ and $t=\circ(b,\tau)$ holds because of $W(\iota_a)(\sigma)=W(\iota_b)(\tau)$. Similarly to the proof of Lemma~\ref{lem:normal-form}, we get
 \begin{equation*}
  a=[\iota_a]^{<\omega}(\supp^W_a(\sigma))=\supp^W_{a\cup b}(W(\iota_a)(\sigma))=\supp^W_{a\cup b}(W(\iota_b)(\tau))=b.
 \end{equation*}
 Since $W(\iota_a)=W(\iota_b)$ is injective, it follows that $W(\iota_a)(\sigma)=W(\iota_b)(\tau)$ implies $\sigma=\tau$ and hence $s=t$. For transitivity we consider $t=\circ(c,\rho)$. If $r\leq_{\T W(X)}s\leq_{\T W(X)}t$ holds because we have $s\leq_{\T W(X)}t'$ for some $t'\in c$, then the induction hypothesis yields $r\leq_{\T W(X)}t'$ and hence $r\leq_{\T W(X)}t$. Now assume that $s=\circ(b,\tau)\leq_{\T W(X)}t$ holds due to
 \begin{equation*}
  W(\iota_b)(\tau)\leq_{W(b\cup c)} W(\iota_c)(\rho).
 \end{equation*}
 Since~$W$ is normal this implies $b\leqf_{\T W(X)}c$, as in the proof of the previous lemma. So if $r\leq_{\T W(X)}s$ holds because we have $r\leq_{\T W(X)}s'$ for some $s'\in b$, then we get
 \begin{equation*}
  r\leq_{\T W(X)}s'\leq_{\T W(X)}t'\qquad\text{for some $t'\in c$}.
 \end{equation*}
 By induction hypothesis this yields $r\leq_{\T W(X)}t'$, which implies $r\leq_{\T W(X)}t$. Finally, assume that $r=\circ(a,\sigma)\leq_{\T W(X)}s$ holds because we have
 \begin{equation*}
  W(\iota_a)(\sigma)\leq_{W(a\cup b)} W(\iota_b)(\tau).
 \end{equation*}
 To conclude $r\leq_{\T W(X)} t$ it suffices to consider the inclusions into the set~$a\cup b\cup c$, which is partially ordered due to the simultaneous induction hypothesis.
\end{proof}

For a suitable formalization of PO-dilators in second order arithmetic, the following has been shown by Freund, Rathjen and Weiermann~\cite{frw-kruskal}: The principle that $\T W(\emptyset)$ is a well partial order for any normal WPO-dilator~$W$ is equivalent to \mbox{$\Pi^1_1$-comprehension}, over $\mathbf{RCA_0}$ together with the chain antichain principle. A fortiori, $\Pi^1_1$-comprehension does also follow from the principle that $\T W(X)$ is a well partial order whenever the same holds for~$X$. The following result shows that the converse implication remains true as well, since the minimal bad sequence argument in its proof can be justified by~$\Pi^1_1$-comprehension. Even though the proof is similar to the one of~\cite[Theorem~3.10]{frw-kruskal}, we provide it for the reader's convenience.

\begin{proposition}\label{prop:derivative-preserves-wpos}
 Consider a normal WPO-dilator $W$. If $X$ is a well partial order, then so is $\T W(X)$.
\end{proposition}
\begin{proof}
 We use Nash-Williams'~\cite{nash-williams63} minimal bad sequence argument. Given a partial order~$Y$, an infinite sequence $y_0,y_1,\ldots\subseteq Y$ is called good if there are indices $i<j$ with $y_i\leq_Y y_j$; otherwise it is called bad. Hence~$Y$ is a well partial order if, and only if, it contains no bad sequence. Aiming at a contradiction, we assume that there is a bad sequence $t_0,t_1,\ldots\subseteq\T W(X)$ while~$X$ a well partial order. We may assume that $t_0,t_1,\ldots$ is minimal, in the sense that $t_0,\dots,t_{i-1},t_i',t_{i+1}',\ldots$ is good whenever we have $h_X(t_i')<h_X(t_i)$. This step requires $\Pi^1_1$-comprehension; a~detailed justification can, for example, be found in the proof of~\cite[Theorem~3.10]{frw-kruskal}. For $i\in\mathbb N$ we now define $a_i\subseteq\T W(X)$ by
 \begin{equation*}
  a_i=\begin{cases}
        a & \text{if $t_i=\circ(a,\sigma)$ for some $\sigma\in W(a)$},\\
        \emptyset & \text{if $t_i$ is of the form $\overline x$}.
       \end{cases}
 \end{equation*}
 Let us show that $Z:=\bigcup\{a_i\,|\,i\in\mathbb N\}\subseteq\T W(X)$ is a well partial order. Assuming the contrary, we get a bad sequence $s_0,s_1,\dots$ in~$Z$. Since each set $a_i$ is finite, there are strictly increasing functions $i,j:\mathbb N\to\mathbb N$ with $s_{i(k)}\in a_{j(k)}$ for all~$k\in\mathbb N$. In particular we get $h_X(s_{i(0)})<h_X(t_{j(0)})$. Since the sequence~$t_0,t_1,\ldots$ was assumed to be minimal, this means that
 \begin{equation*}
  t_0,t_1,\dots,t_{j(0)-1},s_{i(0)},s_{i(1)},s_{i(2)},\ldots\subseteq\T W(X)
 \end{equation*}
 must be good. As $t_0,t_1,\ldots$ and $s_0,s_1,\ldots$ are bad, this is only possible if we have $t_k\leq_{\T W(X)}s_{i(l)}$ for some $k<j(0)$ and $l\in\mathbb N$. In view of $s_{i(l)}\in a_{j(l)}$ we can write $t_{j(l)}=\circ(a_{j(l)},\sigma)$ and conclude $t_k\leq_{\T W(X)}t_{j(l)}$. This inequality contradicts the assumption that~$t_0,t_1,\ldots$ is bad, so that $Z$ must be a well partial order after all. Since~$X$ is a well partial order, the bad sequence~$t_0,t_1,\ldots$ can only have finitely many entries of the form~$\overline x$. Passing to a subsequence, we may assume that all entries have the form $t_i=\circ(a_i,\sigma_i)$. Note that this subsequence is bad but not necessarily minimal; we still have $a_i\subseteq Z$ for any entry of the subsequence. Write $\iota_i:a_i\hookrightarrow Z$ for the inclusions and consider the sequence
 \begin{equation*}
  W(\iota_0)(\sigma_0),W(\iota_1)(\sigma_1),\ldots\subseteq W(Z).
 \end{equation*}
 Since $W$ is a WPO-dilator and $Z$ is a well partial order, we obtain indices $i<j$ with $W(\iota_i)(\sigma_i)\leq_{W(Z)}W(\iota_j)(\sigma_j)$. By factoring $\iota_i=\iota\circ\iota_i'$ into $\iota_i':a_i\hookrightarrow a_i\cup a_j$ and $\iota:a_i\cup a_j\hookrightarrow Z$, one readily deduces $W(\iota_i')(a_i)\leq_{W(a_i\cup a_j)}W(\iota_j')(a_j)$. Due to clause~(iii') of Definition~\ref{def:TW(X)} we get
 \begin{equation*}
  t_i=\circ(a_i,\sigma_i)\leq_{\T W(X)}\circ(a_j,\sigma_j)=t_j.
 \end{equation*}
 So~$t_0,t_1,\ldots$ cannot be bad after all.
\end{proof}

\section{A categorical characterization}

The term systems $\T W(X)$ from the previous section can be hard to handle, both in general arguments and in concrete examples. To resolve this issue, the present section provides a more semantic approach. We begin with a general notion:

\begin{definition}\label{def:Kruskal-fixed-point}
Consider a normal PO-dilator~$W$ and a partial order~$X$. A Kruskal fixed point of $W$ over~$X$ consists of a partial order~$Z$ and functions $\iota:X\to Z$ and $\kappa:W(Z)\to Z$ that satisfy $\rng(\iota)\cap\rng(\kappa)=\emptyset$ and
 \begin{alignat*}{3}
  \iota(x)&\leq_Z\iota(y)&&\quad\Rightarrow\quad&&x\leq_X y\quad\text{(for $x,y\in X$)},\\
  \iota(x)&\leq_Z\kappa(\tau)&&\quad\Leftrightarrow\quad&&\iota(x)\leqf_Z\supp^W_Z(\tau)\quad\text{(for $x\in X$ and $\tau\in W(Z)$)},\\
  \kappa(\sigma)&\not\leq_Z\iota(y)&& &&\text{for all $\sigma\in W(Z)$ and $y\in X$},\\
  \kappa(\sigma)&\leq_Z\kappa(\tau)&&\quad\Leftrightarrow\quad&&\sigma\leq_{W(Z)}\tau\text{ or }\kappa(\sigma)\leqf_Z\supp^W_Z(\tau)\quad\text{(for $\sigma,\tau\in W(Z)$)}.
 \end{alignat*}
\end{definition}

Note that we do not demand that $x\leq_X y$ implies $\iota(x)\leq_Z\iota(y)$. This will become important in the proof of Theorem~\ref{thm:deriv-exists}. The following is justified by Lemma~\ref{lem:normal-form}.

\begin{definition}\label{def:iota-kappa}
 Consider a normal PO-dilator~$W$. For each partial order~$X$ we define functions $\iota_X:X\to\T W(X)$ and $\kappa_X:W(\T W(X))\to\T W(X)$ by stipulating
 \begin{align*}
  \iota_X(x)&=\overline x,\\
  \kappa_X(\sigma)&=\circ(a,\sigma_0)\quad\text{for $\sigma\nf W(\iota_a)(\sigma_0)$}.
 \end{align*}
\end{definition}

Let us verify that $\T W(X)$ has the desired structure:

\begin{theorem}\label{thm:TW(X)-fixed-point}
 We consider a normal PO-dilator~$W$ and a partial order~$X$. The order~$\T W(X)$ and the functions $\iota_X$ and $\kappa_X$ form a Kruskal fixed point of $W$ over~$X$.
\end{theorem}
\begin{proof}
In view of Definition~\ref{def:TW(X)} it is immediate that we have $\rng(\iota_X)\cap\rng(\kappa_X)=\emptyset$, that $\iota_X(x)=\overline x\leq_{\T W(X)}\overline y=\iota_X(y)$ implies (and is indeed equivalent to) $x\leq_X y$, and that $\kappa_X(\sigma)=\circ(a,\sigma_0)\leq_{\T W(X)}\overline y=\iota_X(y)$ is always false. For $\tau\nf W(\iota_b)(\tau_0)$ we also get
\begin{equation*}
 \iota_X(x)=\overline x\leq_{\T W(X)}\circ(b,\tau_0)=\kappa_X(\tau)\quad\Leftrightarrow\quad\iota_X(x)\leqf_{\T W(X)}b=\supp^W_{\T W(X)}(\tau).
\end{equation*}
For the remaining equivalence we need to show
\begin{multline*}
\circ(a,\sigma_0)\leq_{\T W(X)}\circ(b,\tau_0)\,\Leftrightarrow{}\\
{}\Leftrightarrow\, W(\iota_a)(\sigma_0)\leq_{W(\T W(X))}W(\iota_b)(\tau_0)\text{ or }\circ(a,\sigma_0)\leqf_{\T W(X)}b,
\end{multline*}
with $\iota_a:a\hookrightarrow\T W(X)$ and $\iota_b:b\hookrightarrow\T W(X)$. In view of Definition~\ref{def:TW(X)} it suffices to observe that we have
\begin{equation*}
W(\iota_a)(\sigma_0)\leq_{W(\T W(X))}W(\iota_b)(\tau_0)\quad\Leftrightarrow\quad W(\iota_a')(\sigma_0)\leq_{W(a\cup b)}W(\iota_b')(\tau_0),
\end{equation*}
with $\iota_a':a\hookrightarrow a\cup b$ and $\iota_b':b\hookrightarrow a\cup b$. To establish this equivalence one considers the inclusion $\iota:a\cup b\hookrightarrow\T W(X)$ and composes the right side with~$W(\iota)$.
\end{proof}

To obtain a unique characterization, we use the following categorical notion.

\begin{definition}
Consider a normal PO-dilator~$W$ and a partial order~$X$. A Kruskal fixed point $(Z,\iota,\kappa)$ is called initial if any Kruskal fixed point $(Z',\iota',\kappa')$ of $W$ over~$X$ admits a unique quasi embedding $f:Z\to Z'$ with $f\circ\iota=\iota'$ and $f\circ\kappa=\kappa'\circ W(f)$.
\end{definition}

Like all initial objects, initial Kruskal fixed points are unique up to isomorphism.
The following criterion will be very useful.

\begin{theorem}\label{thm:initial-condition}
For a Kruskal fixed point $(Z,\iota,\kappa)$ of a normal PO-dilator~$W$ over a partial order~$X$, the following are equivalent:
\begin{enumerate}[label=(\roman*)]
\item We have $\rng(\iota)\cup\rng(\kappa)=Z$, and $x\leq_X y$ implies $\iota(x)\leq_Z\iota(y)$ for $x,y\in X$. Furthermore, there is a function $h:Z\to\mathbb N$ such that
\begin{equation*}
s\in\supp^W_Z(\sigma)\quad\Rightarrow\quad h(s)<h(\kappa(\sigma))
\end{equation*}
holds for any $s\in Z$ and any $\sigma\in W(Z)$.
\item The Kruskal fixed point $(Z,\iota,\kappa)$ is initial.
\end{enumerate}
\end{theorem}
\begin{proof}
Let us first show that condition~(i) implies~(ii). For $s\in Z$ we define $l(s)\in\mathbb N$ by recursion on $h(s)$, setting
\begin{equation*}
l(\iota(x))=0\qquad\text{and}\qquad l(\kappa(\sigma))=1+\textstyle\sum_{s\in\supp^W_Z(\sigma)}2\cdot l(s).
\end{equation*}
Note that each element of $Z$ is covered by exactly one clause, since Definition~\ref{def:Kruskal-fixed-point} and part~(i) of the present theorem provide $\rng(\iota)\cap\rng(\kappa)=\emptyset$ and $\rng(\iota)\cup\rng(\kappa)=Z$. Now consider another Kruskal fixed point $(Z,\iota',\kappa')$. We first show that there is at most one quasi embedding $f:Z\to Z'$ with $f\circ\iota=\iota'$ and $f\circ\kappa=\kappa'\circ W(f)$. These equations amount to
\begin{align*}
f(\iota(x))&=\iota'(x) && \text{for $x\in X$},\\
f(\kappa(\sigma))&=\kappa'(W(f)(\sigma))=\kappa'(W(f\!\restriction\!a)(\sigma_0)) && \text{for $\sigma\nf W(\iota_a)(\sigma_0)\in W(Z)$},
\end{align*}
where $f\!\restriction\!a=f\circ\iota_a:a\to Z'$ is the restriction of~$f$. Once again, each argument of~$f$ is covered by exactly one of these clauses. From Lemma~\ref{lem:normal-form} we know that $\sigma\nf W(\iota_a)(\sigma_0)$ implies $\supp^W_Z(\sigma)=a$. Now a straightforward induction on $l(s)$ shows that $f(s)$ is uniquely determined. To establish existence we read the above as recursive clauses. We verify
\begin{align*}
r\in Z\quad&\Rightarrow\quad f(r)\in Z',\\
f(s)\leq_{Z'}f(t)\quad&\Rightarrow\quad s\leq_Z t
\end{align*}
by simultaneous induction on~$l(r)$ and~$l(s)+l(t)$. Let us verify the first claim for $r=\kappa(\sigma)$ with $\sigma\nf W(\iota_a)(\sigma_0)$. For $s,t\in a$ we have $l(s)+l(t)<l(r)$. Hence the simultaneous induction hypothesis ensures that $f\!\restriction\!a$ is a quasi embedding. We may thus form~$W(f\!\restriction\!a)$, as needed for the clause that defines the value~$f(r)\in Z'$. Let us now show that~$f$ is a quasi embedding. For $s=\iota(x)$ and $t=\iota(y)$ we see that
\begin{equation*}
 f(s)=f(\iota(x))=\iota'(x)\leq_{Z'}\iota'(y)=f(\iota(y))=f(t)
\end{equation*}
implies~$x\leq_X y$. By the condition in~(i) this implies $s=\iota(x)\leq_Z\iota(y)=t$, as required. For $s=\iota(x)$ and $t=\kappa(\tau)$ with $\tau\nf W(\iota_b)(\tau_0)$, a glance at Definition~\ref{def:Kruskal-fixed-point} reveals that $f(s)=\iota'(x)\leq_{Z'}\kappa'(W(f\!\restriction\!b)(\tau_0))=f(t)$ implies
\begin{equation*}
 f(s)\leqf_{Z'}\supp^W_{Z'}(W(f\!\restriction\!b)(\tau_0))=[f\!\restriction\!b]^{<\omega}(\supp^W_b(\tau_0))=[f]^{<\omega}(b).
\end{equation*}
As $t'\in b=\supp^W_Z(\tau)$ implies $h(t')<h(\kappa(\tau))=h(t)$, the induction hypothesis yields $s\leqf_Z\supp^W_Z(\tau)$, which implies $s=\iota(x)\leq_Z\kappa(\tau)=t$. For $s=\kappa(\sigma)$ and $t=\iota(y)$ it suffices to observe that $f(s)\leq_{Z'}f(t)$ cannot hold. Finally, we consider the case of $s=\kappa(\sigma)$ and $t=\kappa(\tau)$ with $\sigma\nf W(\iota_a)(\sigma_0)$ and $\tau\nf W(\iota_b)(\tau_0)$. If
\begin{equation*}
 f(s)=\kappa'(W(f\!\restriction\!a)(\sigma_0))\leq_{Z'}\kappa'(W(f\!\restriction\!b)(\tau_0))=f(t)
\end{equation*}
holds because of $f(s)\leqf_{Z'}\supp^W_{Z'}(W(f\!\restriction\!b)(\tau_0))$, then one argues as above. Now assume that we have
\begin{equation*}
 W(f\!\restriction\!a)(\sigma_0)\leq_{W(Z')}W(f\!\restriction\!b)(\tau_0).
\end{equation*}
The induction hypothesis ensures that $f\!\restriction\!(a\cup b):a\cup b\to Z'$ is a quasi embedding. Here it is crucial that we argue by induction on $l(s)+l(t)$, not on $h(s)+h(t)$. Let us factor $f\!\restriction\! a=f\!\restriction\!(a\cup b)\circ\iota_a'$ and $f\!\restriction\! b=f\!\restriction\!(a\cup b)\circ\iota_b'$, where $\iota_a':a\hookrightarrow a\cup b$ and $\iota_b':b\hookrightarrow a\cup b$ are the inclusions. Then the last inequality amounts to
\begin{equation*}
 W(f\!\restriction\!(a\cup b))\circ W(\iota_a')(\sigma_0)\leq_{W(Z')}W(f\!\restriction\!(a\cup b))\circ W(\iota_b')(\tau_0).
\end{equation*}
Since $W(f\!\restriction\!(a\cup b))$ is a quasi embedding, we get $W(\iota_a')(\sigma_0)\leq_{W(a\cup b)}W(\iota_b')(\tau_0)$. Now compose both sides with the embedding $W(\iota)$, where $\iota:a\cup b\hookrightarrow Z$ is the inclusion. This yields
\begin{equation*}
 \sigma=W(\iota_a)(\sigma_0)=W(\iota\circ\iota_a')(\sigma_0)\leq_Z W(\iota\circ\iota_b')(\tau_0)=W(\iota_b)(\tau_0)=\tau.
\end{equation*}
The latter implies $s=\kappa(\sigma)\leq_Z\kappa(\tau)=t$, which completes the proof that~(i) implies~(ii). To show that~(ii) implies~(i), we first establish~(i) for the Kruskal fixed point $(\T W(X),\iota_X,\kappa_X)$ from Theorem~\ref{thm:TW(X)-fixed-point}. Any $\circ(a,\sigma_0)\in\T W(X)$ arises as $\kappa_X(\sigma)$ for $\sigma=W(\iota_a)(\sigma_0)$. Here the condition $\supp^W_a(\sigma_0)=a$ from Definition~\ref{def:TW(X)} ensures that $\sigma$ is in normal form. This shows $\rng(\iota_X)\cup\rng(\kappa_X)=\T W(X)$. The requirement that $x\leq_X y$ implies $\iota_X(x)=\overline x\leq_{\T W(X)}\overline y=\iota_X(y)$ is immediate by Definition~\ref{def:TW(X)}. A height function $h_X:\T W(X)\to\mathbb N$ has been defined before the statement of Lemma~\ref{lem:ineq-heights}. For arbitrary elements $\sigma\nf W(\iota_a)(\sigma_0)\in W(\T W(X))$ and $s\in\supp^W_{\T W(X)}(\sigma)=a$, the construction of $h_X$ entails
\begin{equation*}
 h_X(s)<h_X(\circ(a,\sigma_0))=h_X(\kappa_X(\sigma)),
\end{equation*}
just as needed. Since we have already shown that~(i) implies~(ii), we can conclude that $(\T W(X),\iota_X,\kappa_X)$ is an initial Kruskal fixed point of~$W$ over~$X$. If $(Z,\iota,\kappa)$ is any initial Kruskal fixed point as in~(ii), we get an isomorphism $f:Z\to\T W(X)$ with $f\circ\iota=\iota_X$ and $f\circ\kappa=\kappa_X\circ W(f)$. As~$W(f)$ is an isomorphism, the ranges of $\kappa_X$ and $\kappa_X\circ W(f)$ coincide. Hence we get $\T W(X)=\rng(f\circ\iota)\cup\rng(f\circ\kappa)$ and then $Z=\rng(\iota)\cup\rng(\kappa)$. We also learn that $x\leq_X y$ implies
\begin{equation*}
 \iota(x)=f\circ\iota_X(x)\leq_{Z}f\circ\iota_X(y)=\iota(y).
\end{equation*}
Finally, we define $h:Z\to\mathbb N$ by $h(s)=h_X(f(s))$. For $s\in\supp^W_Z(\sigma)$ we have
\begin{equation*}
f(s)\in[f]^{<\omega}(\supp^W_Z(\sigma))=\supp^W_{\T W(X)}(W(f)(\sigma)).
\end{equation*}
We can conclude
\begin{equation*}
 h(s)=h_X(f(s))<h_X(\kappa_X\circ W(f)(\sigma))=h_X(f\circ\kappa(\sigma))=h(\kappa(\sigma)),
\end{equation*}
as required for~(i).
\end{proof}

The following result was shown as part of the previous proof. It is important, because it establishes the existence of initial Kruskal fixed points.

\begin{corollary}\label{cor:fixed-points-exist}
For each normal PO-dilator~$W$ and each partial order~$X$, the Kruskal fixed point $(\T W(X),\iota_X,\kappa_X)$ is initial.
\end{corollary}

For later use we also record the following result.

\begin{lemma}\label{lem:embeddings-preserved}
 Let $(Z,\iota,\kappa)$ be an initial Kruskal fixed point of a normal PO-dilator~$W$ over a partial order~$X$. Consider another Kruskal fixed point~$(Z',\iota',\kappa')$ and the unique quasi embedding $f:Z\to Z'$ with $f\circ\iota=\iota'$ and $f\circ\kappa=\kappa'\circ W(f)$. If $x\leq_X y$ implies $\iota'(x)\leq_{Z'}\iota'(y)$ for all $x,y\in X$, then $f$ is an embedding.
\end{lemma}
\begin{proof}
 Define $l:Z\to\mathbb N$ as in the proof of Theorem~\ref{thm:initial-condition}. In the latter we have used induction on $l(s)+l(t)$ to show that $f(s)\leq_{Z'}f(t)$ implies $s\leq_Z t$. Assuming that $x\leq_X y$ implies $\iota'(x)\leq_{Z'}\iota'(y)$, one can read the given argument in reverse, to show that $s\leq_Z t$ does also imply $f(s)\leq_{Z'}f(t)$.
\end{proof}

So far, the notation $\T W(X)$ has been reserved for the term systems constructed in Definition~\ref{def:TW(X)}. In the following sections we will also use $\T W(X)$ for other initial Kruskal fixed points of $W$ over~$X$. This is harmless, since we have shown that all these fixed points are equivalent.

\section{Kruskal derivatives}\label{sect:kruskal-derivative}

Consider a normal PO-dilator~$W$. As shown in the previous section, each partial order~$X$ gives rise to an initial Kruskal fixed point~$(\T W(X),\iota_X,\kappa_X)$. In the present section we show that the transformation $X\mapsto\T W(X)$ of partial orders can be extended into a normal PO-dilator~$\T W$. More precisely, we will show that there is an essentially unique extension in the sense of the following definition.

\begin{definition}\label{def:kruskal-deriv}
 A Kruskal derivative of a normal PO-dilator $W$ is tuple $(\T W,\iota,\kappa)$ that consists of a normal PO-dilator $\T W$ and two families of functions
 \begin{equation*}
  \iota_X:X\to\T W(X)\quad\text{and}\quad\kappa_X:W(\T W(X))\to\T W(X)
 \end{equation*}
 indexed by the partial order~$X$, such that the following properties are satisfied:
 \begin{enumerate}[label=(\roman*)]
  \item The tuple $(\T W(X),\iota_X,\kappa_X)$ is an initial Kruskal fixed point of~$W$ over~$X$, for each partial order~$X$.
  \item We have $\iota_Y\circ f=\T W(f)\circ\iota_X$ and $\T W(f)\circ\kappa_X=\kappa_Y\circ W(\T W(f))$, for any quasi embedding $f:X\to Y$ between partial orders.
  \end{enumerate}
\end{definition}

Let us begin by proving existence:

\begin{theorem}\label{thm:deriv-exists}
 Each normal PO-dilator has a Kruskal derivative.
\end{theorem}
\begin{proof}
Consider a normal PO-dilator~$W$. For each partial order~$X$, Corollary~\ref{cor:fixed-points-exist} provides an initial Kruskal fixed point~$(\T W(X),\iota_X,\kappa_X)$ of~$W$ over~$X$. Given a quasi embedding $f:X\to Y$, it is easy to see that $(\T W(Y),\iota_Y\circ f,\kappa_Y)$ is a Kruskal fixed point of~$W$ over~$X$ as well. Since $(\T W(X),\iota_X,\kappa_X)$ is initial, there is a unique quasi embedding $\T W(f):\T W(X)\to\T W(Y)$ with
 \begin{equation*}
  \iota_Y\circ f=\T W(f)\circ\iota_X\quad\text{and}\quad\T W(f)\circ\kappa_X=\kappa_Y\circ W(\T W(f)).
 \end{equation*}
 In order to obtain a Kruskal derivative $(\T W,\iota,\kappa)$, it suffices to turn~$\T W$ into a normal PO-dilator. To show that $\T W$ is a functor, one checks that $\T W(g)\circ\T W(f)$ satisfies the equations that characterize~$\T W(g\circ f)$. If $f:X\to Y$ is an embedding, then $x\leq_X y$ implies $\iota_Y\circ f(x)\leq_{\T W(Y)}\iota_Y\circ f(y)$, since $\iota_Y$ must satisfy the condition from part~(i) of Theorem~\ref{thm:initial-condition}. Hence Lemma~\ref{lem:embeddings-preserved} ensures that $\T W(f)$ is again an embedding, as required in part~(i) of Definition~\ref{def:po-dilator}. It remains to exhibit suitable support functions
 \begin{equation*}
  \supp^{\T W}_X:\T W(X)\to[X]^{<\omega}.
 \end{equation*}
 In view of Theorem~\ref{thm:initial-condition} we can recursively define
 \begin{align*}
  \supp^{\T W}_X(\iota_X(x))&=\{x\},\\
  \supp^{\T W}_X(\kappa_X(\sigma))&=\bigcup\{\supp^{\T W}_X(s)\,|\,s\in\supp^W_{\T W(X)}(\sigma)\}.
 \end{align*}
 To show naturality, one verifies
 \begin{equation*}
  \supp^{\T W}_Y(\T W(f)(s))=[f]^{<\omega}(\supp^{\T W}_X(s))
 \end{equation*}
 by induction on $h_X(s)$, where $h_X:\T W(X)\to\mathbb N$ is as in part~(i) of Theorem~\ref{thm:initial-condition}. To satisfy the support condition from part~(ii) of Definition~\ref{def:po-dilator}, we need to establish
 \begin{equation*}
  \supp^{\T W}_Y(s)\subseteq\rng(f)\quad\Rightarrow\quad s\in\rng(\T W(f))
 \end{equation*}
 for an embedding~$f:X\to Y$ (recall that the converse implication is automatic). We use induction on~$h_Y(s)$. For $s=\iota_Y(y)$ we see that $\{y\}=\supp^{\T W}_Y(s)\subseteq\rng(f)$ yields $y=f(x)$ for some $x\in X$. This entails
 \begin{equation*}
  s=\iota_Y\circ f(x)=\T W(f)\circ\iota_X(x)\in\rng(\T W(f)).
 \end{equation*}
 Now consider $s=\kappa_Y(\sigma)$. For any $s'\in\supp^W_{\T W(Y)}(\sigma)$ we have $h_Y(s')<h_Y(s)$ and
 \begin{equation*}
  \supp^{\T W}_Y(s')\subseteq\supp^{\T W}_Y(s)\subseteq\rng(f),
 \end{equation*}
 so that the induction hypothesis yields $s'\in\rng(\T W(f))$. Thus we get
 \begin{equation*}
  \supp^W_{\T W(Y)}(\sigma)\subseteq\rng(\T W(f)).
 \end{equation*}
Now the support condition for the PO-dilator~$W$ yields $\sigma=W(\T W(f))(\sigma_0)$ for some $\sigma_0\in W(\T W(X))$. We then obtain
\begin{equation*}
 s=\kappa_Y\circ W(\T W(f))(\sigma_0)=\T W(f)\circ\kappa_X(\sigma_0)\in\rng(\T W(f)),
\end{equation*}
 as required. It remains to show that the PO-dilator~$\T W$ is normal. We verify
 \begin{equation*}
  s\leq_{\T W(X)}t\quad\Rightarrow\quad\supp^{\T W}_X(s)\leqf_X\supp^{\T W}_X(t)
 \end{equation*}
 by induction on~$h_X(s)+h_X(t)$. If we have $s=\iota_X(x)\leq_{\T W(X)}\iota_X(y)=t$, then we must have $x\leq_X y$ and hence
 \begin{equation*}
  \supp^{\T W}_X(s)=\{x\}\leqf_X\{y\}=\supp^{\T W}_X(t).
 \end{equation*}
 Now consider the case of an inequality $s\leq_{\T W(X)}\kappa_X(\tau)=t$ that holds because we have $s\leq_{\T W(X)}t'$ for some $t'\in\supp^W_{\T W(X)}(\tau)$. In view of $h_X(t')<h_X(t)$, the induction hypothesis yields
 \begin{equation*}
  \supp^{\T W}_X(s)\leqf_X\supp^{\T W}_X(t')\subseteq\supp^{\T W}_X(t).
 \end{equation*}
 Finally, assume that $s=\kappa_X(\sigma)\leq_{\T W(X)}\kappa_X(\tau)=t$ holds due to $\sigma\leq_{W(\T W(X))}\tau$. Since~$W$ is normal, we get $\supp^W_{\T W(X)}(\sigma)\leqf_{\T W(X)}\supp^W_{\T W(X)}(\tau)$. Given an arbitrary $s'\in\supp^W_{\T W(X)}(\sigma)$, we may then pick a $t'\in\supp^W_{\T W(X)}(\tau)$ with $s'\leq_{\T W(X)}t'$. By induction hypothesis we get
 \begin{equation*}
  \supp^{\T W}_X(s')\leqf_X\supp^{\T W}_X(t')\subseteq\supp^{\T W}_X(t).
 \end{equation*}
 Since $s'\in\supp^W_{\T W(X)}(\sigma)$ was arbitrary, this establishes
 \begin{equation*}
  \supp^{\T W}_X(s)=\bigcup\{\supp^{\T W}_X(s')\,|\,s'\in\supp^W_{\T W(X)}(\sigma)\}\leqf_X\supp^{\T W}_X(t),
 \end{equation*}
 as required.
\end{proof}

Let us highlight some of the information that is implicit in the previous proof:

\begin{remark}
In order to construct a Kruskal derivative of a specific PO-dilator, one can follow the proof of Theorem~\ref{thm:deriv-exists}. The latter shows that we only need to find a family of initial Kruskal fixed points. The extension into a Kruskal derivative is then automatic. In particular, the fact that one obtains a normal PO-dilator does not need to be verified in each specific case. Also observe that the functor $\T W$ was uniquely determined by the initial Kruskal fixed points~$\T W(X)$. The choice of support functions is also unique (as for any PO-dilator), since $\supp^{\T W}_X(s)$ must be the smallest set $a\subseteq X$ with $s\in\rng(\T W(\iota_a))$, where $\iota_a:a\hookrightarrow\T W(X)$ is the inclusion: In one direction, the support condition from part~(ii) of Definition~\ref{def:po-dilator} ensures $s\in\rng(\T W(\iota_a))$ for $a=\supp^{\T W}_X(s)$. In the other direction, naturality entails that $s=\T W(\iota_a)(s_0)$ yields
\begin{equation*}
 a\subseteq[\iota_a]^{<\omega}(\supp^{\T W}_a(s_0))=\supp^{\T W}_X(\T W(\iota_a)(s_0))=\supp^{\T W}_X(s).
\end{equation*}
We have not included this information in the statement of Theorem~\ref{thm:deriv-exists}, because a more general uniqueness result will be shown below.
\end{remark} 

To prepare our uniqueness result, we recall that two PO-dilators $(V,\supp^V)$ and $(W,\supp^W)$ are equivalent if there is a natural isomorphism~$\eta:V\Rightarrow W$ of functors. It may also seem reasonable to demand
\begin{equation*}
 \supp^W_X\circ\eta_X=\supp^V_X
\end{equation*}
for any partial order~$X$. However, the latter turns out to be automatic. Girard~\cite{girard-pi2} has shown that this is the case for any natural transformation between dilators of linear orders (cf.~also \cite[Lemma~2.17]{freund-rathjen_derivatives}, which is closer to our notation). One can check that the proof remains valid for partial orders. In the case of an isomorphism, the argument is particularly simple: Given $\sigma\in W(X)$, we invoke Lemma~\ref{lem:normal-form} to write $\sigma=W(\iota_a)(\sigma_0)$ with $a=\supp^V_X(\sigma)$. We then get
\begin{multline*}
 \supp^W_X\circ\eta_X(\sigma)=\supp^W_X(\eta_X\circ V(\iota_a)(\sigma_0))=\supp^W_X(W(\iota_a)\circ\eta_a(\sigma_0))=\\
 =[\iota_a]^{<\omega}(\supp^W_a(\eta_a(\sigma_0)))\subseteq\rng(\iota_a)=a=\supp^V_X(\sigma).
\end{multline*}
By applying the same argument to the inverse of $\eta$, we also get
\begin{equation*}
 \supp^V_X(\sigma)=\supp^V_X\circ\eta_X^{-1}(\eta_X(\sigma))\subseteq\supp^W_X(\eta_X(\sigma))=\supp^W_X\circ\eta_X(\sigma).
\end{equation*}
The following result shows that Kruskal derivatives are essentially unique.

\begin{theorem}\label{thm:derivs-unique}
 For any two Kruskal derivatives $(\T^0W,\iota^0,\kappa^0)$ and $(\T^1W,\iota^1,\kappa^1)$ of a normal PO-dilator~$W$, there is a natural isomorphism $\eta:\T^0W\Rightarrow\T^1W$ such that we have $\eta_X\circ\iota^0_X=\iota^1_X$ and $\eta_X\circ\kappa^0_X=\kappa^1_X\circ W(\eta_X)$ for any partial order~$X$.
\end{theorem}
\begin{proof}
 For each order~$X$, the fact that $(\T^0W(X),\iota^0_X,\kappa^0_X)$ and $(\T^1W(X),\iota^1_X,\kappa^1_X)$ are initial Kruskal fixed points of~$W$ over~$X$ implies that there is an isomorphism $\eta_X:\T^0W(X)\to\T^1W(X)$ with $\eta_X\circ\iota^0_X=\iota^1_X$ and $\eta_X\circ\kappa^0_X=\kappa^1_X\circ W(\eta_X)$. It remains to show that the resulting family~$\eta$ is natural. Given a quasi embedding $f:X\to Y$ between partial orders, we show
 \begin{equation*}
  \eta_Y\circ\T^0W(f)(s)=\T^1W(f)\circ\eta_X(s)
 \end{equation*}
 by induction on $h^0_X(s)$, where $h^0_X:\T^0W(X)\to\mathbb N$ is as in part~(i) of Theorem~\ref{thm:initial-condition}. To cover elements of the form $s=\iota^0_X(x)$, it suffices to observe
 \begin{equation*}
  \eta_Y\circ\T^0W(f)\circ\iota^0_X=\eta_Y\circ\iota^0_Y\circ f=\iota^1_Y\circ f=\T^1W(f)\circ\iota^1_X=\T^1W(f)\circ\eta_X\circ\iota^0_X.
 \end{equation*}
 Given an element $s=\kappa^0_X(\sigma)$, we invoke Lemma~\ref{lem:normal-form} to write $\sigma\nf W(\iota_a)(\sigma_0)$, where $\iota_a:a\hookrightarrow\T^0W(X)$ is the inclusion. For any  element $s'\in a=\supp^W_{\T^0W(X)}(\sigma)$ we have $h^0_X(s')<h^0_X(s)$, so that the induction hypothesis yields
 \begin{equation*}
  \eta_Y\circ\T^0W(f)\circ\iota_a=\T^1W(f)\circ\eta_X\circ\iota_a.
 \end{equation*}
 We can deduce
 \begin{multline*}
  \eta_Y\circ\T^0W(f)(s)=\eta_Y\circ\T^0W(f)\circ\kappa^0_X(\sigma)=\eta_Y\circ\kappa^0_Y\circ W(\T^0W(f))(\sigma)=\\
  =\kappa^1_Y\circ W(\eta_Y)\circ W(\T^0W(f)\circ\iota_a)(\sigma_0)=\kappa^1_Y\circ W(\T^1W(f))\circ W(\eta_X\circ\iota_a)(\sigma_0)=\\
  =\T^1W(f)\circ\kappa^1_X\circ W(\eta_X)(\sigma)=\T^1W(f)\circ\eta_X\circ\kappa^0_X(\sigma)=\T^1 W(f)\circ\eta_X(s),
 \end{multline*}
 as required.
\end{proof}

Given a normal PO-dilator~$W$, we will write $\T W$ for ``its" Kruskal derivative, even though the latter is only determined up to isomorphism. The following is an immediate consequence of Proposition~\ref{prop:derivative-preserves-wpos}.

\begin{corollary}\label{cor:derivative-wpo}
If $W$ is a normal WPO-dilator, then its Kruskal derivative~$\T W$ is a normal WPO-dilator as well.
\end{corollary}

In the following sections we will consider iterated Kruskal derivatives. To ensure that the iterations are essentially unique, we now show that equivalent PO-dilators have equivalent Kruskal derivatives.

\begin{proposition}\label{prop:isomorphic-iterates}
 Consider a natural isomorphism $\eta:V\Rightarrow W$ between normal PO-dilators. If $(\T W,\iota,\kappa)$ is a Kruskal derivative of $W$, then $(\T W,\iota,\kappa\circ\eta)$ is a Kruskal derivative of~$V$, where $\kappa\circ\eta$ is defined by $(\kappa\circ\eta)_X=\kappa_X\circ\eta_{\T W(X)}$.
\end{proposition}
\begin{proof}
 It is straightforward to verify that $(\T W(X),\iota_X,(\kappa\circ\eta)_X)$ is an initial Kruskal fixed point of~$V$ over~$X$, for any partial order~$X$. To provide a representative part of the verification, we show that $(\kappa\circ\eta)_X(\sigma)\leq_{\T W(X)}(\kappa\circ\eta)_X(\tau)$ is equivalent to
 \begin{equation*}
  \sigma\leq_{V(\T W(X))}\tau\quad\text{or}\quad(\kappa\circ\eta)_X(\sigma)\leqf_{\T W(X)}\supp^V_{\T W(X)}(\tau),
 \end{equation*}
 as required by Definition~\ref{def:Kruskal-fixed-point}. Since $(\T W(X),\iota_X,\kappa_X)$ is a Kruskal fixed point of~$W$ over~$X$, the same definition entails that $(\kappa\circ\eta)_X(\sigma)\leq_{\T W(X)}(\kappa\circ\eta)_X(\tau)$ is equivalent to the disjunction of $\eta_{\T W(X)}(\sigma)\leq_{W(\T W(X))}\eta_{\T W(X)}(\tau)$ and
 \begin{equation*}
 (\kappa\circ\eta)_X(\sigma)\leqf_{\T W(X)}\supp^W_{\T W(X)}(\eta_{\T W(X)}(\tau)).
 \end{equation*}
 The first disjunct is equivalent to $\sigma\leq_{V(\T W(X))}\tau$. To relate the second disjuncts, it suffices to recall that we have
 \begin{equation*}
  \supp^W_{\T W(X)}(\eta_{\T W(X)}(\tau))=\supp^V_{\T W(X)}(\tau).
 \end{equation*}
 To conclude that $(\T W,\iota,\kappa\circ\eta)$ is a Kruskal derivative in the sense of Definition~\ref{def:kruskal-deriv}, we compute
 \begin{multline*}
  \T W(f)\circ(\kappa\circ\eta)_X=\T W(f)\circ\kappa_X\circ\eta_{\T W(X)}=\kappa_Y\circ W(\T W(f))\circ\eta_{\T W(X)}=\\
  =\kappa_Y\circ\eta_{\T W(Y)}\circ V(\T W(f))=(\kappa\circ\eta)_Y\circ V(\T W(f))
 \end{multline*}
 for an arbitrary quasi embedding $f:X\to Y$.
\end{proof}

\section{The gap orders as PO-dilators}\label{sect:gap-orders-dilators}

In the present section we give a recursive definition of the set of finite trees with labels in~$\{0,\dots,n-1\}$, ordered according to Friedman's gap condition. We also show that one obtains a normal PO-dilator if one relativizes the gap orders to a given partial order~$X$. This prepares the reconstruction of Friedman's gap condition in the following section.

As a preparation, we give a more precise account of finite multisets: Let us write~$X^{<\omega}$ for the set of finite sequences $\langle x_0,\dots,x_{n-1}\rangle$ with entries~$x_i\in X$. Say that two sequences $\langle x_0,\dots,x_{m-1}\rangle$ and $\langle y_0,\dots,y_{n-1}\rangle$ in $X^{<\omega}$ are equivalent if, and only if, there is a bijective function $h:\{0,\dots,m-1\}\to\{0,\dots,n-1\}$ such that we have $x_i=y_{h(i)}$ for all $i<m=n$. We write $[x_0,\dots,x_{n-1}]$ for the equivalence class of~$\langle x_0,\dots,x_{n-1}\rangle$ with respect to this equivalence relation. From $[x_0,\dots,x_{n-1}]$ one can recover the multiplicity but not the order of the entries. The quotient set
\begin{equation*}
 M(X)=\{[x_0,\dots,x_{n-1}]\,|\,\langle x_0,\dots,x_{n-1}\rangle\in X^{<\omega}\}
\end{equation*}
is called the set of finite multisets with elements from~$X$. We declare that
\begin{equation*}
 [x_0,\dots,x_{m-1}]\leq_{M(X)}[y_0,\dots,y_{n-1}]
\end{equation*}
holds if, and only if, there is an injection $g:\{0,\dots,m-1\}\to\{0,\dots,n-1\}$ such that we have $x_i\leq_X y_{g(i)}$ for all $i<m$. One can check that this is well defined and yields a partial order on $M(X)$ (for antisymmetry, use induction on the number of elements). Higman's lemma entails that $M(X)$ is a well partial order if the same holds for~$X$. Given a (quasi) embedding $f:X\to Y$, one can define a (quasi) embedding $M(f):M(X)\to M(Y)$ by setting
\begin{equation*}
 M(f)([x_0,\dots,x_{n-1}])=[f(x_0),\dots,f(x_{n-1})].
\end{equation*}
A family of functions $\supp^M_X:M(X)\to[X]^{<\omega}$ can be given by
\begin{equation*}
 \supp^M_X([x_0,\dots,x_{n-1}])=\{x_0,\dots,x_{n-1}\}.
\end{equation*}
It is straightforward to check that this turns $M$ into a normal WPO-dilator in the sense of Definitions~\ref{def:po-dilator} and~\ref{def:normal}. 

Given a partial order~$X$, the underlying set of the partial order~$\ot_n(X)$ consists of the finite trees with labels in $\{0,\dots,n-1\}\cup X$, where labels from~$X$ may only occur at the leafs. More formally, this set admits the following recursive description:

\begin{definition}
Given a number~$n\in\mathbb N$ and a partial order~$X$, we generate a set $\ot_n(X)$ by the following recursive clauses:
\begin{enumerate}[label=(\roman*)]
\item For each $x\in X$ we have an element $\overline x\in\ot_n(X)$.
\item Whenever we have constructed an element $\sigma=[t_0,\dots,t_{m-1}]\in M(\ot_n(X))$, we add an element $i\star\sigma\in\ot_n(X)$ for each natural number~$i<n$.
\end{enumerate}
Let us also define
\begin{equation*}
\ot_n^-(X)=\{\overline x\,|\,x\in X\}\cup\{0\star\sigma\,|\,\sigma\in M(\ot_n(X))\}\subseteq\ot_n(X),
\end{equation*}
provided that we have~$n>0$.
\end{definition}

We define height functions $h_X^n:\ot_n(X)\to\mathbb N$ by the recursive clauses
\begin{equation*}
h_X^n(\overline x)=0,\qquad h_X^n(i\star[t_0,\dots,t_{m-1}])=\max(\{0\}\cup\{h_X^n(t_k)+1\,|\,k<m\}).
\end{equation*}
The following definition decides $s\leq_{\ot_n(X)}t$ by recursion on $h_X^n(s)+h_X^n(t)$.

\begin{definition}\label{def:gap-order}
To define a binary relation $\leq_{\ot_n(X)}$ on the set $\ot_n(X)$ we stipulate
\begin{alignat*}{3}
\overline x&\leq_{\ot_n(X)}t\quad&&\Leftrightarrow\quad&&
\begin{cases}
\text{either }t=\overline y\text{ with }x\leq_X y,\\
\text{or }t=j\star[t_0,\dots,t_{m-1}]\text{ and }\overline x\leq_{\ot_n(X)}t_l\text{ for some $l<m$},
\end{cases}\\
i\star\sigma&\leq_{\ot_n(X)}t\quad&&\Leftrightarrow\quad&&
\begin{cases}
t=i\star\tau\text{ with }\sigma\leq_{M(\ot_n(X))}\tau,\text{ or }t=j\star[t_0,\dots,t_{m-1}]\\
\text{with }j\geq i\text{ and }i\star\sigma\leq_{\ot_n(X)}t_l\text{ for some $l<m$}.
\end{cases}
\end{alignat*}
In the case of $n>0$, we define $\leq_{\ot_n^-(X)}$ as the restriction of $\leq_{\ot_n(X)}$ to $\ot_n^-(X)$.
\end{definition}

A straightforward induction shows
\begin{equation*}
s\leq_{\ot_n(X)}t\quad\Rightarrow\quad h_X^n(s)\leq h_X^n(t).
\end{equation*}
Similarly to the proof of Proposition~\ref{prop:fixed-points-partial-order}, one can deduce that $\leq_{\ot_n(X)}$ is a partial order on~$\ot_n(X)$. In the introduction we have given the usual definition of Friedman's gap condition for embeddings of $n$-trees. The following shows that the recursive clauses from Definition~\ref{def:gap-order} yield the same result. We assume that isomorphic \mbox{$n$-trees} are identified.

\begin{proposition}\label{prop:gap-orders}
The partial order~$\ot_n(\emptyset)$ is isomorphic to the set of $n$-trees, ordered according to Friedman's strong gap condition.
\end{proposition}
\begin{proof}
For~$s=i\star[s(0),\dots,s(k-1)]\in\ot_n(\emptyset)$ we recursively define $T_s$ as the $n$-tree with root label~$i$ and immediate subtrees $T_{s(0)},\dots,T_{s(k-1)}$. It is clear that this yields a bijection. By induction on $h_X^n(s)+h_X^n(t)$ one can show that $s\leq_{\ot_n(\emptyset)}t$ holds if, and only if, there is an embedding $f:T_s\to T_t$ that satisfies Friedman's gap condition. An inequality
\begin{equation*}
s=i\star\sigma=i\star[s(0),\dots,s(k-1)]\leq_{\ot_n(\emptyset)}i\star[t(0),\dots,t(m-1)]=i\star\tau=t
\end{equation*}
that holds because of $\sigma\leq_{M(\ot_n(X))}\tau$ corresponds to an embedding $f:T_s\to T_t$ that maps the root to the root. Indeed, the inequalities $s(j)\leq_{\ot_n(\emptyset)}t(l_j)$ that witness $\sigma\leq_{M(\ot_n(X))}\tau$ correspond to the restrictions
\begin{equation*}
f_j=f\!\restriction\! T_{s(j)}:T_{s(j)}\to T_{t(l_j)}\subseteq T_t.
\end{equation*}
At this point it is crucial that we consider the strong gap condition: Writing $\rt(T)$ for the root of~$T$, the gap below $f_j(\rt(T_{s(j)}))\in T_{t(l_j)}$ corresponds to the gap between $f(\rt(T_s))$ and $f(\rt(T_{s(j)}))$ in $T_s$. An inequality
\begin{equation*}
s=i\star\sigma\leq_{\ot_n(X)}j\star[t_0,\dots,t_{m-1}]=t
\end{equation*}
that holds because of $j\geq i$ and $s\leq_{\ot_n(X)}t_l$ with $l<m$ corresponds to an embedding $f:T_s\rightarrow T_t$ with range contained in $T_{t(l)}\subseteq T_t$. The condition $j\geq i$ accounts for the fact that $\rt(T_t)$ lies in the gap below $f(\rt(T_s))$ in $T_t$ but not in $T_{t(l)}$.
\end{proof}

Our next goal is to extend $\ot_n$ and $\ot_{n+1}^-$ into PO-dilators.

\begin{definition}
Given a quasi embedding $f:X\to Y$ between partial orders, we define $\ot_n(f):\ot_n(X)\to\ot_n(Y)$ by the recursive clauses
\begin{equation*}
\ot_n(f)(\overline x)=\overline{f(x)},\qquad\ot_n(f)(i\star[t_0,\dots,t_{m-1}])=i\star[\ot_n(f)(t_0),\dots,\ot_n(f)(t_{m-1})].
\end{equation*}
For $n>0$ we observe that $\ot_n(f)$ restricts to $\ot_n^-(f):\ot_n^-(X)\to\ot_n^-(Y)$. We also define a family of functions $\supp^{\ot_n}_X:\ot_n(X)\to[X]^{<\omega}$ by stipulating
\begin{equation*}
\supp^{\ot_n}_X(\overline x)=\{x\},\qquad\supp^{\ot_n}_X(i\star[t_0,\dots,t_{m-1}])=\bigcup\{\supp^{\ot_n}_X(t_l)\,|\,l<m\}.
\end{equation*}
We will write $\supp^{\ot_n^-}_X$ for the restriction of $\supp^{\ot_n}_X$ to $\ot_n^-(X)$.
\end{definition}

Let us verify that we obtain the desired structure:

\begin{proposition}\label{prop:gap-dilators}
The previous definition yields normal PO-dilators~$\ot_n$ and $\ot_{n+1}^-$.
\end{proposition}
\begin{proof}
Given a quasi embedding~$f$, an easy induction on $h_X^n(s)+h_X^n(t)$ shows
\begin{equation*}
\ot_n(f)(s)\leq_{\ot_n(Y)}\ot_n(f)(t)\quad\Rightarrow\quad s\leq_{\ot_n(X)}t.
\end{equation*}
If $f$ is an embedding, then the converse implication holds as well. Also by induction, one readily checks that $\ot_n$ is a functor and that $\supp^{\ot_n}$ is a natural transformation. To conclude that $\ot_n$ a PO-dilator, one needs to establish the support condition from part~(ii) of Definition~\ref{def:po-dilator}. By induction on~$s$, one can indeed show
\begin{equation*}
\supp^{\ot_n}_Y(s)\subseteq\rng(f)\quad\Rightarrow\quad s\in\rng(\ot_n(f))
\end{equation*}
for $s\in\ot_n(Y)$, where $f:X\to Y$ is an embedding (recall that the converse implication is automatic). To see that $\ot_{n+1}^-$ does also satisfy the support condition, one should observe that $\ot_{n+1}(f)(s)\in\ot_{n+1}^-(Y)$ implies $s\in\ot_{n+1}^-(X)$. To establish the normality condition from Definition~\ref{def:normal}, one verifies
\begin{equation*}
s\leq_{\ot_n(X)}t\quad\Rightarrow\quad\supp^{\ot_n}_X(s)\leqf_X\supp^{\ot_n}_X(t)
\end{equation*}
by induction on $h_X^n(s)+h_X^n(t)$.
\end{proof}

Corollary~\ref{cor:T_n-wpo} below will establish the stronger result that $\ot_n$ and $\ot_{n+1}^-$ are normal WPO-dilators. In view of Proposition~\ref{prop:gap-orders} this implies that the trees with Friedman's gap condition form a well partial order. To prove this fact one needs iterated applications of~$\Pi^1_1$-comprehension.

\section{Reconstructing the gap condition}

In the introduction we have sketched the reconstruction of Friedman's gap condition in terms of iterated Kruskal derivatives. The reader may wish to recall steps~(1) to~(4) from the introduction, which describe a recursive construction of normal WPO-dilators~$\ot_n$ and $\ot_{n+1}^-$. We now show that the latter are unique up to natural isomorphism: Inductively, we may assume that this is the case for $\ot_n$ and hence for $M\circ\ot_n$ (see below for the composition of PO-dilators). Theorem~\ref{thm:derivs-unique} and Proposition~\ref{prop:isomorphic-iterates} ensure that $\ot_{n+1}^-$, which is the Kruskal derivative of $M\circ\ot_n$, is unique up to natural isomorphism as well. Finally, the same holds for the composition $\ot_{n+1}=\ot_n\circ\ot_{n+1}^-$. The recursive construction via steps~(1) to~(4) may seem at odds with the ad hoc definition of $\ot_n$ and $\ot_{n+1}^-$ in the previous section. However, this objection is easily resolved: In the following we will show that the PO-dilators~$\ot_n$ and $\ot_{n+1}^-$ from the previous section are related as specified by steps~(1) to~(4) from the introduction. Due to uniqueness, this means that our ad hoc definition coincides with the result of the recursive construction.

Let us first observe that the normal WPO-dilator~$\ot_0$ from the previous section is equivalent to the identity functor on the category of partial orders. Hence step~(1) from the introduction is satisfied, at least up to natural isomorphism. In Proposition~\ref{prop:gap-orders} we have shown that $\ot_n(\emptyset)$ is isomorphic to the set of $n$-trees with Friedman's strong gap condition, as claimed by step~(4). Our next goal is to verify step~(3) from the introduction, which requires that $\ot_{n+1}$ is equivalent to $\ot_n\circ\ot_{n+1}^-$. Let us first discuss the composition of dilators in general: To compose PO-dilators $V$ and~$W$ one first takes their composition as functors. In order to obtain a PO-dilator, one defines a family of functions $\supp^{V\circ W}_X:V\circ W(X)\to[X]^{<\omega}$ by setting
\begin{equation*}
\supp^{V\circ W}_X(\sigma)=\bigcup\{\supp^W_X(s)\,|\,s\in\supp^V_{W(X)}(\sigma)\}.
\end{equation*}
If $V$ and $W$ are WPO-dilators, then so is $V\circ W$. One readily checks that $V\circ W$ is normal if the same holds for~$V$ and~$W$. As explained in Section~\ref{sect:kruskal-derivative}, two PO-dilators are equivalent if they are equivalent as functors. One can verify that $V\circ W$ is equivalent to $V'\circ W'$ if $V$ is equivalent to $V'$ and $W$ is equivalent to $W'$. To realize step~(3), we will show that the following defines an equivalence.

\begin{definition}
For each partial order~$X$, we define~$\pi^n_X:\ot_n\circ\ot_{n+1}^-(X)\to\ot_{n+1}(X)$ by the recursive clauses
\begin{equation*}
\pi^n_X(\overline t)=t,\qquad\pi^n_X(i\star[s_0,\dots,s_{m-1}])=(i+1)\star[\pi^n_X(s_0),\dots,\pi^n_X(s_{m-1})],
\end{equation*}
where the first clause relies on the inclusion $\ot_{n+1}^-(X)\subseteq\ot_{n+1}(X)$.
\end{definition}

Intuitively speaking, an element of $\ot_n\circ\ot_{n+1}^-(X)$ is a finite tree with labels from $\{0,\dots,n-1\}\cup\ot_{n+1}^-(X)$, where the labels from $\ot_{n+1}^-(X)$ can only occur at leafs. The function $\pi^n_X$ increases the labels from $\{0,\dots,n-1\}$ and ``unravels" the leaf labels. Hence the leafs of $s\in\ot_n\circ\ot_{n+1}^-(X)$ correspond to the minimal nodes of $\pi^n_X(s)\in\ot_{n+1}(X)$ that have a label in~$\{0\}\cup X$. It is interesting to observe that the inverse of $\pi^n_X$ is similar to the transformation $T\mapsto T^*$ from~\cite[Section~4]{simpson85}. Let us verify the promised result:

\begin{proposition}\label{prop:pi-iso}
The family $\pi^n:\ot_n\circ\ot_{n+1}^-\Rightarrow\ot_{n+1}$ is a natural isomorphism.
\end{proposition}
\begin{proof}
In order to show that $\pi^n_X$ is surjective we verify $t\in\rng(\pi^n_X)$ by induction on~$t\in\ot_{n+1}(X)$. If $t$ is of the form $\overline x$ or $0\star\sigma$, then we have $t\in\ot_{n+1}^-(X)$, which yields $\overline t\in\ot_n\circ\ot_{n+1}^-(X)$ and $t=\pi^n_X(\overline t)\in\rng(\pi^n_X)$. Now consider an element of the form $t=(i+1)\star[t_0,\dots,t_{m-1}]$, with $i+1<n+1$ and $t_l\in\ot_{n+1}(X)$ for $l<m$. Inductively we get $t_l=\pi^n_X(s_l)$, which yields $i\star[s_0,\dots,s_{m-1}]\in\ot_n\circ\ot_{n+1}^-(X)$ and
\begin{equation*}
t=(i+1)\star[t_0,\dots,t_{m-1}]=\pi^n_X(i\star[s_0,\dots,s_{m-1}])\in\rng(\pi^n_X).
\end{equation*}
To conclude that $\pi^n_X$ is an isomorphism, we show
\begin{equation*}
s\leq_{\ot_n\circ\ot_{n+1}^-(X)}t\quad\Leftrightarrow\quad \pi^n_X(s)\leq_{\ot_{n+1}(X)}\pi^n_X(t)
\end{equation*}
by induction on $h^n_{\ot_{n+1}^-(X)}(s)+h^n_{\ot_{n+1}^-(X)}(t)$. For $s=\overline{s'}$ and $t=\overline{t'}$ it suffices to invoke Definition~\ref{def:gap-order}. Now consider $s=\overline{ s'}$ and $t=j\star[t_0,\dots,t_{m-1}]$. Inductively we get
\begin{equation*}
s\leq_{\ot_n\circ\ot_{n+1}^-(X)}t\quad\Leftrightarrow\quad\pi^n_X(s)\leq_{\ot_{n+1}(X)}\pi^n_X(t_l)\text{ for some $l<m$}.
\end{equation*}
Note that $\pi^n_X(s)=s'\in\ot_{n+1}^-(X)$ must be of the form $\overline x$ or $0\star\sigma$. In view of $j+1\neq 0$ and $j+1\geq 0$, the right side of the previous equivalence is thus equivalent to
\begin{equation*}
\pi^n_X(s)\leq_{\ot_{n+1}(X)}(j+1)\star[\pi^n_X(t_0),\dots,\pi^n_X(t_{m-1})]=\pi^n_X(t).
\end{equation*}
For $s=i\star[s_0,\dots,s_{k-1}]$ and $t=\overline{t'}$ we cannot have $s\leq_{\ot_n\circ\ot_{n+1}^-(X)}t$. We also see
\begin{equation*}
\pi^n_X(s)=(i+1)\star[\pi^n_X(s_0),\dots,\pi^n_X(s_{k-1})]\not\leq_{\ot_{n+1}(X)} t'=\pi^n_X(t),
\end{equation*}
since an inequality would require $t'=j\star[t_0,\dots,t_{m-1}]$ with $j\geq i+1$, in contrast to~$t'\in\ot_{n+1}^-(X)$. For $s=i\star[s_0,\dots,s_{k-1}]$ and $t=j\star[t_0,\dots,t_{m-1}]$ the claim is readily deduced from the induction hypothesis (due to $i\geq j\Leftrightarrow i+1\geq j+1$). To complete the proof we verify the naturality property
\begin{equation*}
\pi^n_Y\circ(\ot_n\circ\ot_{n+1}^-)(f)(t)=\ot_{n+1}(f)\circ\pi^n_X(t),
\end{equation*}
arguing by induction on $t\in\ot_n\circ\ot_{n+1}^-(X)$. For $t=\overline s$ we compute
\begin{multline*}
\pi^n_Y\circ(\ot_n\circ\ot_{n+1}^-)(f)(t)=\pi^n_Y\circ\ot_n(\ot_{n+1}^-(f))(\overline s)=\pi^n_Y(\overline{\ot_{n+1}^-(f)(s)})=\\
=\ot_{n+1}^-(f)(s)=\ot_{n+1}(f)(s)=\ot_{n+1}(f)\circ\pi^n_X(t).
\end{multline*}
The induction step for $t=j\star[t_0,\dots,t_{m-1}]$ is straightforward.
\end{proof}

The following lemma will be needed below. Intuitively, the equivalence says that a tree with root label $0$ can be embedded into another tree if, and only if, it can be embedded into a subtree with root label~$0$. This is true because the gap condition below a node with label~$0$ is automatic.

\begin{lemma}\label{lem:zero-subtrees}
We have
 \begin{align*}
s\leq_{\ot_{n+1}(X)}\pi^n_X(t)\quad&\Leftrightarrow\quad s\leqf_{\ot_{n+1}^-(X)}\supp^{\ot_n}_{\ot_{n+1}^-(X)}(t)
\end{align*}
for all $s\in\ot_{n+1}^-(X)$ and all $t\in\ot_n\circ\ot_{n+1}^-(X)$.
\end{lemma}
\begin{proof}
We establish the claim by induction on~$t$. For $t=\overline{t'}$ it suffices to observe $\pi^n_X(t)=t'$ and $\supp^{\ot_n}_{\ot_{n+1}^-(X)}(t)=\{t'\}$. To prove the claim for $t=j\star[t_0,\dots,t_{m-1}]$, we recall a step from the previous proof: For $s\in\ot_{n+1}^-$ we have observed
\begin{equation*}
s\leq_{\ot_{n+1}(X)}\pi^n_X(t)\quad\Leftrightarrow\quad s\leq_{\ot_{n+1}(X)}\pi^n_X(t_l)\text{ for some $l<m$}.
\end{equation*}
Together with
\begin{equation*}
\supp^{\ot_n}_{\ot_{n+1}^-(X)}(t)=\bigcup\{\supp^{\ot_n}_{\ot_{n+1}^-(X)}(t_l)\,|\,l<m\},
\end{equation*}
this reduces the claim to the induction hypothesis. 
\end{proof}

To complete the reconstruction of the gap condition, it remains to realize step~(2) from the introduction. For this purpose we show that $\ot_{n+1}^-$ is a Kruskal derivative of $M\circ\ot_n$, where $M$ is the finite multiset dilator from the beginning of Section~\ref{sect:gap-orders-dilators}. In view of Definition~\ref{def:kruskal-deriv}, we introduce the following objects:

\begin{definition}
For any partial order~$X$ we define a function $\iota^n_X:X\to\ot_{n+1}^-(X)$ by setting $\iota^n_X(x)=\overline x$. To define $\kappa^n_X:M\circ\ot_n\circ\ot_{n+1}^-(X)\to\ot_{n+1}^-(X)$ we stipulate
\begin{equation*}
\kappa^n_X([s_0,\dots,s_{m-1}])=0\star[\pi^n_X(s_0),\dots,\pi^n_X(s_{m-1})],
\end{equation*}
for $s_0,\dots,s_{m-1}\in\ot_n\circ\ot_{n+1}^-(X)$. We will write $\iota^n$ and $\kappa^n$ for the families of functions $\iota^n_X$ and $\kappa^n_X$ that are indexed by the partial order~$X$.
\end{definition}

Let us now prove the central result of our reconstruction:

\begin{theorem}
For any number $n\in\mathbb N$, the tuple $(\ot_{n+1}^-,\iota^n,\kappa^n)$ is a Kruskal derivative of the normal PO-dilator~$M\circ\ot_n$. 
\end{theorem}
\begin{proof}
From Proposition~\ref{prop:gap-dilators} we know that $\ot_{n+1}^-$ is a normal PO-dilator. It remains to verify conditions~(i) and~(ii) from Definition~\ref{def:kruskal-deriv}. Let us begin by showing that $(\ot_{n+1}^-(X),\iota^n_X,\kappa^n_X)$ is an initial Kruskal fixed point of $M\circ\ot_n$ over $X$, for each partial order~$X$. Invoking the fact that $\pi^n_X:\ot_n\circ\ot_{n+1}^-(X)\to\ot_{n+1}(X)$ is surjective, we see that $\ot_{n+1}^-(X)$ is the disjoint union of $\rng(\iota^n_X)$ and $\rng(\kappa^n_X)$, as required for Definition~\ref{def:Kruskal-fixed-point} and Theorem~\ref{thm:initial-condition}. In view of Definition~\ref{def:gap-order} we also have
\begin{gather*}
\iota^n_X(x)=\overline x\leq_{\ot_{n+1}^-(X)}\overline y=\iota^n_X(y)\quad\Leftrightarrow\quad x\leq_X y,\\
\kappa^n_X([s_0,\dots,s_{k-1}])=0\star[\pi^n_X(s_0),\dots,\pi^n_X(s_{k-1})]\not\leq_{\ot_{n+1}^-(X)}\overline y=\iota^n_X(y).
\end{gather*}
To verify the remaining conditions from Definition~\ref{def:Kruskal-fixed-point}, we observe that the support of an element $\tau=[t_0,\dots,t_{m-1}]\in M\circ\ot_n\circ\ot_{n+1}^-(X)$ is given by
\begin{multline*}
\supp^{M\circ\ot_n}_{\ot_{n+1}^-(X)}(\tau)=\bigcup\{\supp^{\ot_n}_{\ot_{n+1}^-(X)}(t)\,|\,t\in\supp^M_{\ot_n\circ\ot_{n+1}^-(X)}(\tau)\}=\\
=\bigcup\{\supp^{\ot_n}_{\ot_{n+1}^-(X)}(t_l)\,|\,l<m\}.
\end{multline*}
For $s\in\ot_{n+1}^-(X)$ we can thus invoke Lemma~\ref{lem:zero-subtrees} to get
\begin{equation*}
s\leqf_{\ot_{n+1}^-(X)}\supp^{M\circ\ot_n}_{\ot_{n+1}^-(X)}(\tau)\quad\Leftrightarrow\quad s\leq_{\ot_{n+1}(X)}\pi^n_X(t_l)\text{ for some $l<m$}.
\end{equation*}
Writing $\sigma=[s_0,\dots,s_{k-1}]$ and $\tau=[t_0,\dots,t_{m-1}]$, we now see that the second condition from Definition~\ref{def:Kruskal-fixed-point} requires that
\begin{equation*}
\iota^n_X(x)=\overline x\leq_{\ot_{n+1}^-(X)}0\star[\pi^n_X(t_0),\dots,\pi^n_X(t_{m-1})]=\kappa^n_X(\tau)
\end{equation*}
holds if, and only if, we have $\overline x\leq_{\ot_{n+1}^-(X)}\pi^n_X(t_l)$ for some $l<m$. This is true according to Definition~\ref{def:gap-order}. The last condition from Definition~\ref{def:Kruskal-fixed-point} requires that
\begin{equation*}
\kappa^n_X(\sigma)=0\star[\pi^n_X(s_0),\dots,\pi^n_X(s_{k-1})]\leq_{\ot_{n+1}^-(X)}0\star[\pi^n_X(t_0),\dots,\pi^n_X(t_{m-1})]=\kappa^n_X(\tau)
\end{equation*}
is equivalent to the disjunction
\begin{equation*}
\sigma\leq_{M\circ\ot_n\circ\ot_{n+1}^-(X)}\tau\quad\text{or}\quad\kappa^n_X(\sigma)\leq_{\ot_{n+1}(X)}\pi^n_X(t_l)\text{ for some $l<m$}.
\end{equation*}
To reduce this to Definition~\ref{def:gap-order} it suffices to note that we have
\begin{equation*}
 \sigma\leq_{M\circ\ot_n\circ\ot_{n+1}^-(X)}\tau\,\Leftrightarrow\,[\pi^n_X(s_0),\dots,\pi^n_X(s_{k-1})]\leq_{M\circ\ot_{n+1}(X)}[\pi^n_X(t_0),\dots,\pi^n_X(t_{m-1})],
\end{equation*}
 since $\pi^n_X$ is an embedding. Now recall the function $h_X^{n+1}:\ot_{n+1}(X)\to\mathbb N$ that was specified before the statement of Defintion~\ref{def:gap-order} above. We will also write $h_X^{n+1}$ for the restriction of this function to $\ot_{n+1}^-(X)\subseteq\ot_{n+1}(X)$. In order to apply Theorem~\ref{thm:initial-condition}, we need to establish
 \begin{equation*}
  s\in\supp^{M\circ\ot_n}_{\ot_{n+1}^-(X)}(\tau)\quad\Rightarrow\quad h_X^{n+1}(s)<h_X^{n+1}(\kappa_X^n(\tau))
 \end{equation*}
 for $s\in\ot_{n+1}^-(X)$ and $\tau\in M\circ\ot_n\circ\ot_{n+1}^-(X)$. So assume we have $s\in\supp^{M\circ\ot_n}_{\ot_{n+1}^-(X)}(\tau)$ with $\tau=[t_0,\dots,t_{m-1}]$. By the above we get $s\in\supp^{\ot_n}_{\ot_{n+1}^-(X)}(t_l)$ for some~$l<m$. Then Lemma~\ref{lem:zero-subtrees} yields $s\leq_{\ot_{n+1}(X)}\pi_X^n(t_l)$. As observed after Definition~\ref{def:gap-order}, this implies $h_X^{n+1}(s)\leq h_X^{n+1}(\pi_X^n(t_l))$ and hence
 \begin{equation*}
  h_X^{n+1}(s)<h_X^{n+1}(0\star[\pi^n_X(t_0),\dots,\pi^n_X(t_{m-1})])=h^{n+1}_X(\kappa^n_X(\tau)).
 \end{equation*}
 We have now verified all conditions from Definition~\ref{def:Kruskal-fixed-point} and Theorem~\ref{thm:initial-condition}, which shows that $(\ot_{n+1}^-(X),\iota^n_X,\kappa^n_X)$ is an initial Kruskal fixed point of $M\circ\ot_n$ over~$X$. To conclude that $(\ot_{n+1}^-,\iota^n,\kappa^n)$ is a Kruskal derivative of $M\circ\ot_n$, it remains to establish condition~(ii) from Definition~\ref{def:kruskal-deriv}. Given a quasi embedding~$f:X\to Y$, we first compute
 \begin{equation*}
  \iota^n_Y\circ f(x)=\overline{f(x)}=\ot_{n+1}(f)(\overline x)=\ot_{n+1}^-(f)\circ\iota^n_X(x).
 \end{equation*}
 For $\tau=[t_0,\dots,t_{m-1}]\in M\circ\ot_n\circ\ot_{n+1}^-(X)$ we also get
 \begin{multline*}
  \kappa^n_Y\circ(M\circ\ot_n)(\ot_{n+1}^-(f))(\tau)=\kappa^n_Y([(\ot_n\circ\ot_{n+1}^-)(f)(t_0),\dots,(\ot_n\circ\ot_{n+1}^-)(f)(t_{m-1})])=\\
  \begin{aligned}
  &=0\star[\pi^n_Y\circ(\ot_n\circ\ot_{n+1}^-)(f)(t_0),\dots,\pi^n_Y\circ(\ot_n\circ\ot_{n+1}^-)(f)(t_{m-1})]=\\
  &=0\star[\ot_{n+1}(f)\circ\pi^n_X(t_0),\dots,\ot_{n+1}(f)\circ\pi^n_X(t_{m-1})]=\\
  &=\ot_{n+1}(f)(0\star[\pi^n_X(t_0),\dots,\pi^n_X(t_{m-1})])=\ot_{n+1}^-(f)\circ\kappa^n_X(\tau),
  \end{aligned}
 \end{multline*}
 just as required by Definition~\ref{def:kruskal-deriv}.
 \end{proof}
 
 As mentioned in the introduction, we can draw the following conclusion. In view of Proposition~\ref{prop:gap-orders}, the corollary implies Friedman's result that the gap condition induces a well partial order on the set of finite trees with labels from~$\{0,\dots,n-1\}$.
 
 \begin{corollary}\label{cor:T_n-wpo}
  The normal PO-dilators $\ot_n$ and $\ot_{n+1}^-$ preserve well partial orders (which means that they are normal WPO-dilators), for each number~$n\in\mathbb N$.
 \end{corollary}
 \begin{proof}
  We argue by induction on~$n$. Due to $\ot_0(X)\cong X$ it is clear that $\ot_0$ preserves well partial orders. If $\ot_n$ is a normal WPO-dilator, then so is $M\circ\ot_n$. Since $\ot_{n+1}^-$ is the Kruskal derivative of $M\circ\ot_n$, Corollary~\ref{cor:derivative-wpo} implies that it is also a normal WPO-dilator. In view of Proposition~\ref{prop:pi-iso}, the same holds for $\ot_{n+1}\cong\ot_n\circ\ot_{n+1}^-$.
 \end{proof}

\bibliographystyle{amsplain}
\bibliography{Iterated_collapsing}

\end{document}